\documentclass[12pt,a4paper]{amsart}
\usepackage{amsfonts}
\usepackage{textcomp}
\usepackage{amsthm}
\usepackage{amsmath}
\usepackage{amscd}
\usepackage[latin2]{inputenc}
\usepackage{t1enc}
\usepackage[mathscr]{eucal}
\usepackage{indentfirst}
\usepackage{graphicx}
\usepackage{graphics}
\usepackage{pict2e}
\theoremstyle{plain}
\newtheorem{Th}{Theorem}%[section]

\newtheorem{Pro}[Th]{Proposition}
 \theoremstyle{definition}
\newtheorem{Ex}[Th]{Example}

\newtheorem{Rem}[Th]{Remark}
\newtheorem{?}[Th]{Problem}

\newcommand{\Z}{\mathbb{Z}}
\newcommand{\F}{\mathbb{F}}
\newcommand{\C}{\mathbb{C}}

\DeclareMathOperator{\rk}{rank}
\begin{document}

\title[On the  gcd of the value of two polynomials]
{On the  greatest common divisor of the value of two polynomials}

\author[P.\ E.\ Frenkel]{P\'eter E.\ Frenkel}

\address{E\"{o}tv\"{o}s  University \\ Department of Algebra and Number Theory 
 \\ P\'{a}zm\'{a}ny P\'{e}ter s\'{e}t\'{a}ny 1/c \\H-1117 Budapest, Hungary \\ and R\'enyi Institute of Mathematics, Hungarian Academy of Sciences \\ 13-15 Re\'altanoda utca \\ H-1053 Budapest, Hungary}
\email{frenkelp265@gmail.com}

\thanks{Research of the first author  is partially supported by ERC Consolidator Grant 648017, by MTA R\'enyi Lend\"ulet Groups and Graphs research group,  and by the Hungarian National Research, Development and Innovation Office -- NKFIH, OTKA grants no.\ K109684 and K104206.}

\author[J.\ Pelik\'an]{J\'ozsef Pelik\'an}

\address{E\"{o}tv\"{o}s  University \\ Department of Algebra and Number Theory
\\ P\'{a}zm\'{a}ny P\'{e}ter s\'{e}t\'{a}ny 1/c \\ H-1117 Budapest, Hungary }
\email{pelikan@cs.elte.hu}

 \subjclass[2010]{}


\begin{abstract} We show that if  two monic polynomials with integer coefficients have square-free resultant, then all positive divisors of the resultant arise as the greatest common divisor of the values of the two polynomials at a suitable integer.
\end{abstract}

\maketitle

Throughout this paper, %$A$ denotes a principal ideal domain, and
 $f,g\in \Z[x]$ are monic polynomials with integer coefficients:
 \begin{equation}\label{f}f(x)=a_0x^k+a_1x^{k-1}+\dots +a_k\end{equation} and
  \begin{equation}\label{g}g(x)=b_0x^l+b_1x^{l-1}+\dots +b_l,\end{equation} where $a_0=b_0=1$.
  Our interest is in the range of the greatest common divisor $\gcd(f(n),g(n))$ as $n$ varies in the ring $\Z$ of integers. Such  gcd's  can behave in intriguing ways.

\begin{Ex}\label{pl}
%Let $A=\Z$.
(a) A problem in a Hungarian mathematics competition in 2015 %~\cite{OKTV}
  asked for the range of $\gcd\left(n^2+3, (n+1)^2+3\right)$. The answer is $\{1,13\}$. The gcd is 1 for $n=1,\dots ,5$ but is  13 for $n=6$.

(b) The Prime Glossary page~\cite{pg} explaining  the ``law of small numbers'' of R.\ K.\ Guy~\cite{G} points out that the
$\gcd\left(n^{17}+9, (n+1)^{17}+9\right)$ is 1 for  $n=1,\dots, N-1$, but is greater than $1$  for $n=N$, where $$N=8424432925592889329288197322308900672459420460792433.$$ This number $N$ has 52 digits, and the gcd for $n=N$
is the 52-digit prime $$p=8936582237915716659950962253358945635793453256935559.$$
\end{Ex}

Turning to the general case, let $r=R(f,g)\in \Z$ be the \it resultant \rm of the two polynomials. Recall that, by definition, $r$ is the determinant of the \it Sylvester matrix \rm \begin{equation}\label{Sylv}M=\begin{pmatrix} a_0 & a_1 & \dots & a_k & & \\&a_0 & a_1 & \dots & a_k &\\  & & \dots  & \dots &\dots &\dots \\&&&a_0 & a_1 & \dots & a_k \\ b_0 & b_1 & \dots & b_l & & \\&b_0 & b_1 & \dots & b_l &\\  & & \dots  & \dots &\dots &\dots \\&&&b_0 & b_1 & \dots & b_l
\end{pmatrix}\end{equation} of the two polynomials. Note that $M$ is an $(l+k)$-square matrix; the first $l$ rows are built from the coefficients of $f$, and the last $k$ rows
 are built from the coefficients of $g$, padded with zeros.

 The most widely applied fact about the resultant is that it  is zero if and only if the two polynomials have a common complex root, or, equivalently, a non-constant common divisor in $\C[x]$. This holds true even if the coefficients are arbitrary complex numbers. In our case, however, the coefficients are integers. In this setting, the resultant is zero if and only if the two polynomials have  a non-constant common divisor in $\Z[x]$. %has a further --- related, but different --- property:
 %we have $\phi f+\psi g=r$ with suitable $\phi,\psi\in \Z[x]$, see e.g.\ \cite[Section V.10]{L}.
  %This leads to

  We start with  two  easy observations relating the resultant $r$ to the gcd of the polynomial values.

\begin{Pro}\label{div} (a) For any integer $n$, $\gcd(f(n), g(n))$ divides $r$.

(b) % Assume $r\ne 0$.
%Then, a
As a function of $n$, the value $\gcd(f(n), g(n))$ is periodic with period  $r$.
\end{Pro}
Note that $r$ can be zero. By definition, any function is periodic with period 0.
\begin{proof}
(a) Let $d=\gcd(f(n), g(n))$. Each coordinate of the column vector $$M\cdot(n^{k+l-1}, n^{k+l-2}, \dots, n,1)^\top$$ is divisible by either $f(n)$ or $g(n)$, and therefore by $d$. Thus, the last column of $M$ is congruent modulo $d$ to a linear combination, with integral coefficients,  of the previous  columns. It follows that $r=\det M\equiv 0\mod d$, as claimed.

%We have   $\phi (n) f(n)+\psi (n)g(n)=r$ with $\phi(n), \psi(n)\in \Z$, and the claim follows.

(b) We have $f(n+r)\equiv f(n)$ and  $g(n+r)\equiv g(n)$   mod $r$. %., whence $$\gcd(f(n+r), r)=\gcd(f(n),r).$$
% The same holds with $g$ in place of $f$.
 It follows that $$\gcd(f(n+r), g(n+r), r)=\gcd(f(n), g(n), r).$$ % value of $f(n)$ does not change if we replace $n$ by $n+r$. Thus, the set of divisors of $r$ that divide
In view of statement (a), the third argument can be omitted from the gcd on both sides, proving statement (b). %it suffices to prove that the mod $r$ value of the gcd is periodic with period $r$. This follows from the well-known fact that the mod $r$ values of $f(n)$ and $g(n)$ are periodic with period $r$.
\end{proof}

 Recall that $r=0$ if and only if $f$ and $g$ have a non-constant common divisor $h$ in the ring $\Z[x]$. In this case,  $\gcd(f(n), g(n))$ is divisible by $h(n)$ for all $n$ and therefore has an infinite range and no nonzero period.

Do  all nonnegative divisors of $r$  arise as $\gcd(f(n),g(n))$ with suitable integral $n$?  In particular, does  $|r|$ itself arise as such a gcd?   %Are  all periods of the gcd divisible by $r$?
  Not necessarily.

\begin{Ex}\label{3}
Let $f(x)=g(x)=x^2+x+1$. Then $r=0$, so all integers divide $r$, but not all nonnegative integers arise as $\gcd(f(n), g(n))=n^2+n+1$. In fact, no even numbers arise. In particular, 0 itself does not arise.
\end{Ex}

%In relation to Example\ref{4},

What if we assume $r\ne 0$? The answer %to the questions raised before Example~\ref{3}
 is still no.

\begin{Ex}\label{4}%Let $A=\Z$.
Let $f(x)=x^2-1$ and $g(x)=x^2+1$. Then $r=4$,  but the range of $\gcd(f(n), g(n))$ is $\{1,2\}$.
\end{Ex}

This example also shows  that when $r\ne 0$, $|r|$ need not be the smallest positive period of
$\gcd(f(n), g(n))$. In Example~\ref{4},  we have $r=4$, but the smallest positive period is 2.

Our main result, Theorem~\ref{main} below, says that when $r$ is square-free, Proposition~\ref{div}(a)  is the only restriction on the values attained by the gcd, and the smallest positive period of the gcd is $|r|$.

For this, we shall need
a  basic fact about integer matrices: they can be brought to \it Smith normal form. \rm For any matrix $M$ with integral entries, there exist matrices $U$ and $V$, also with integral entries and invertible over $\Z$, such that $UMV$ is a diagonal matrix with diagonal entries $d_1$, $d_2$, \dots, where the so-called \it invariant factors
\rm $d_i$ satisfy $d_i|d_{i+1}$ for all $i$. See Smith's original paper \cite{S}, or see,
e.g., \cite[Section 5.3]{AW} for a textbook presentation. Note that $U$ and $V$, being invertible over $\Z$, are necessarily square matrices with determinant $\pm 1$. If $M$ is also square, it follows that \begin{equation}\label{det}\prod d_i=\det (UMV)=\pm\det M.\end{equation}

In the proof of our main result, we shall have to leave the realm of polynomials with integer coefficients and consider polynomials over the field $\F_p$ of prime cardinality  $p$.  Given two polynomials $f$ and $g$ over any field $F$, of degree $k$ and $l$ respectively, with coefficients as in \eqref{f} and \eqref{g}, their
 Sylvester matrix $M$ is   defined by the formula \eqref{Sylv}.
  We shall need

  \begin{Th} \cite[Theorem 1.19]{J}  The \emph{corank} (or \emph{kernel dimension}) $k+l-\rk M$ of $M$ over $F$ equals the degree of the gcd of the two  polynomials $f$ and $g$ as elements of the polynomial ring $F[x]$.
  \end{Th}

  For two proofs of this well-known fact, the reader may consult \cite{J}. As this is an Internet reference, and we were unable to find  a textbook or journal reference, we include a  third proof.

  \begin{proof} Let us identify the vector space $F^{k+l}$ with the vector space of polynomials of degree less than $k+l$.  Let any such polynomial correspond to the list of its coefficients, starting with the coefficient of $x^{k+l-1}$ and ending with the constant term.

  Under this correspondence, the row space of the Sylvester matrix $M$ is identified with the set of polynomials of the form $\phi f+\psi g$, where $\phi, \psi\in F[x]$ have degree less than $l$ and $k$, respectively. Any polynomial of this form is divisible by $\gcd(f,g)$. Conversely, any polynomial that is divisible by $\gcd(f,g)$ and has degree less than $k+l$ is in the row space. To see this, we first write such a polynomial as %it suffices to show that $\gcd(f,g)$ can be written as
  $\phi_0 f+\psi_0 g$, where we know nothing about the degree of $\phi_0, \psi_0\in F[x]$, but then we write $\phi_0=qg+\phi$ with $\phi$ of degree less than $l$, and we define $\psi=qf+\psi_0$.  Then $\phi_0 f+\psi_0 g=\phi f+\psi g$; moreover, this polynomial and $\phi f$ both have degree less than $k+l$, whence so does $\psi g$, showing that $\psi$ has degree less than $k$.

  The rank of $M$ is the dimension of the row space. The theorem follows.
    \end{proof}

%We shall need two lemmas.

%\begin{Lemma}
We are now ready for the main result of this paper.

\begin{Th}\label{main} Let $f$ and $g$ be monic polynomials with integer coefficients.
 Assume that their resultant $r$ is square-free. Then all positive divisors of $r$  arise as $\gcd(f(n),g(n))$ with suitable integral $n$. Moreover, any $d|r$ arises exactly $\prod (p-1)$ times in each period of length $|r|$, where the product is taken over all (positive)
 prime divisors $p$ of $r/d$. In particular, $|r|$ itself arises once.
\end{Th}

\begin{proof} Let $\mathcal P$ be the set of all %(positive)
prime divisors of $r$, so that $$r=\pm\prod_{p\in\mathcal P}p.$$ We shall prove that for all subsets $\mathcal S$ of $\mathcal P$, the product $d=\prod_{p\in\mathcal S}p$  arises as $\gcd(f(n),g(n))$ for  a suitable integer $n$; moreover, in each period of length $|r|$, it arises exactly  $\prod_{p\in \mathcal P-\mathcal S}(p-1)$ times. %That is, there exists an $a\in A$ such that for each $p\in \mathcal S$, both $f(a)$ and $g(a)$ are divisible by $p$, but this is never the case for $p\in\mathcal P-\mathcal S$.

% It suffices to prove the existence of an $a_p\in A$ f
For each $p\in \mathcal P$, the $\gcd(f(n),g(n),p)$ is periodic with period $p$.  It suffices to prove that in each period of length $p$, this gcd is $p$ exactly once. % such that for each $p\in \mathcal S$, both $f(a_p)$ and $g(a_p)$ are divisible by $p$, but this is never the case for $p\in\mathcal P-\mathcal S$.
Indeed, the Chinese remainder theorem will then finish the proof: in each period of length $|r|$, the integers $n$ such that $\gcd(f(n), g(n))=d$ can be found by specifying their value mod $p$ for each $p\in\mathcal P$. For each $p\in\mathcal S$, there is a unique possibility for $n$ mod $p$,  and for each $p\in \mathcal P-\mathcal S$, there are $p-1$ possibilities. %provide an $a\in A$ with $a\equiv a_p \mod p$ for all $p\in\mathcal P$, and this $a$ will have the desired property above.

% Let $F_p$ denote the field $A/(p)$.

It suffices to prove that for any prime $p\in \mathcal P$, the polynomials $f$ and $g$, when viewed mod $p$,  have a  %common root in $\F_p$, but do not simultaneously vanish on all of $\F_p$. In fact, we shall prove that they have a
unique common root in $\F_p$; equivalently, the gcd of $f$ and $g$ as elements of $\F_p[x]$ has a unique root in $\F_p$.  In fact, we shall prove that this gcd is a polynomial of degree exactly 1.

%  It is well known that over a field, the degree of the gcd of two  polynomials is the corank (i.e., kernel dimension) of the Sylvester matrix built from their coefficients. Thus, i
It suffices to prove that the mod $p$ corank of the Sylvester matrix $M$ of $f$ and $g$ is 1. But the determinant of $M$ over $\Z$ is $r$, which is divisible by $p$ but not by $p^2$. Now $M$ can be brought to Smith normal form, and from \eqref{det}, we see that
the last invariant factor $d_{k+l}$ is divisible by $p$, but the previous one is not. The mod $p$ corank of the diagonal matrix $UMV$, and therefore also of $M$, is 1, as claimed.
\end{proof}

\begin{Rem}
When $|r|$ is prime,  %$\gcd(f(a), g(a))$  is periodic with period $r$; moreover,
the gcd is $|r|$ for $n$ in a  unique residue class mod $r$ and is 1 for all other $n$. This sheds some light on the seemingly peculiar behavior
in Example~\ref{pl}, since $r=13$ for (a) and $r=p$ for (b).
\end{Rem}

When $r$ is not square-free, we know very little about the range of the gcd. At least, we can give  a sufficient condition for 1 to appear in the range. This condition, however, is not necessary; see Example~\ref{4}. %From now on, primes are assumed to be positive.

\begin{Pro}\label{pp} Let $f$ and $g$ be monic polynomials with integer coefficients and resultant $r$.

(a) Suppose that $p$ is prime and  $r$ is not divisible by $p^p$. Then there exists an integer $n$ such that  $\gcd(f(n), g(n))$ is not divisible by $p$.

(b) If $r$ has no divisor of the form $p^p$ with $p$ prime, then there exists an integer $n$ such that $f(n)$ and $g(n)$ are coprime.
\end{Pro}

\begin{proof} (a)  Again we exploit the fact that $r=\pm d_1\cdots d_{k+l}$, where the $d_i$ are the invariant factors of the Sylvester matrix $M$. Since $d_i|d_{i+1}$ for all $i$, and $p^p\not | r$, it follows that at most the last $p-1$ invariant factors $d_i$ can be divisible by $p$. In other words, the mod $p$ corank of $M$ is less than $p$, so the degree of the gcd of $f$ and $g$ as elements of $\F_p[x]$ is less than $p$, and therefore this gcd cannot vanish as a function $\F_p\to\F_p$. But this gcd can be written as $\phi f+\psi g$ with $\phi, \psi\in \F_p[x]$, so it follows that $f$ and $g$ cannot both vanish as functions $\F_p\to\F_p$.

(b) For all prime divisors $p$ of $r$, we can use statement (a) to get an integer $n_p$ such that $\gcd(f(n_p), g(n_p))$ is not divisible by $p$. The Chinese remainder theorem gives us an integer $n$ such that $n\equiv n_p\mod p$ for all $p$. This $n$ will have the desired property.
\end{proof}

\begin{Rem}
Throughout this paper, we have studied two monic polynomials over the ring $\Z$ of integers. However, $\Z$ can be replaced by an arbitrary principal ideal domain $A$. Our results and their proofs remain valid, with trivial modifications.

For example, %the statement that a function $G$ defined on $A$  is periodic with period $r\in A$
Proposition~\ref{div}(b) should be interpreted as saying that $(f(n),g(n))=(f(n'), g(n'))$ whenever $n, n'\in A$ and $r|n-n'$ in $A$. Note that this is an equality of ideals of $A$.

In this general setting, the conclusion of Theorem~\ref{main} is replaced by the following.  There exist constants $c_{P}\in A$, one for each  prime ideal   $P$ containing $r$, such that for any divisor $d$ of $r$, and any $n\in A$,  we have  $(f(n),g(n))=(d)$
if and only if $n-c_{P}\in P$ for each $ P$ containing $d$ but  $n-c_{P}\not\in P$ for each $ P$ that does not contain $d$.
Such elements $n$ exist for any divisor $d$ of $r$. When $d=r$, they form a  coset $c+(r)$.

The $p^p$ in Proposition~\ref{pp} should be interpreted as $p^{|A/(p)|}$. This can be $p^\infty$, which, by definition, divides only 0.
\end{Rem}

\section*{Acknowledgements} We are grateful to the Editorial Board of the \textsc{Monthly} and to the two unnamed referees of this paper for many useful comments.

\end{document}